\documentclass{amsart}

\usepackage{amsmath,amssymb,amsthm,amsfonts,mathrsfs}
\usepackage[top=1in, bottom=1.25in, left=1.25in, right=1.25in]{geometry}
\usepackage{color}
\usepackage[colorlinks=true,
linkcolor=blue,
anchorcolor=blue,
citecolor=red
]{hyperref}

\newtheorem{theorem}{Theorem}[section]

\newtheorem{lemma}[theorem]{Lemma}

\newtheorem{proposition}[theorem]{Proposition}
\theoremstyle{definition}
\newtheorem{definition}[theorem]{Definition}
\theoremstyle{remark}
\newtheorem{remark}[theorem]{Remark}

\newcommand{\Z}{\mathbb{Z}}

\newcommand{\Q}{\mathbb{Q}}

\newcommand{\A}{\mathbb{A}}
\renewcommand{\P}{\mathbb{P}}

\newcommand{\cA}{\mathcal{A}}

\newcommand{\bs}{\backslash}

\newcommand{\wt}{\widetilde}

\newcommand{\of}[1]{\left(#1\right)}
\newcommand{\set}[1]{\left\{#1\right\}}

\newcommand{\spmat}[1]{\left(\begin{smallmatrix}#1\end{smallmatrix}\right)}

\newcommand{\adots}{\mathinner{\mkern2mu%
		\raisebox{0.1em}{.}\mkern2mu\raisebox{0.4em}{.}%
		\mkern2mu\raisebox{0.7em}{.}\mkern1mu}}

\title[Projection operators]{On an identity for the projection operators of automorphic forms}
\author[B. Wang]{Biao Wang}
\address{Department of Mathematics, State University of New York at Buffalo, Buffalo, NY 14260-2900, USA}
\email{bwang32@buffalo.edu}
\date{\today}
\makeatletter
\@namedef{subjclassname@2020}{\textup{}2020 Mathematics Subject Classification}
\makeatother
\subjclass[2020]{11F70}
\keywords{Projection operator; Converse theorems; Fourier-Whittaker expansions}

\begin{document}
\begin{abstract}
	In this note, we revisit an identity that Miller and Schmid showed in their article on a general Voronoi summation formula for $GL(n,\Z)$ in 2009. For the proof, we mainly follow Cogdell and Piatetski-Shapiro's ideas in their work on converse theorems for $GL(n)$.
\end{abstract}

\maketitle
\section{Introduction}

The Voronoi summation formula is a basic tool in the study of analytic number theory. For example, it is a basic ingredient in the $GL(2)$ and $GL(3)$ subconvexity problem \cite{Munshi2018}. Miller and Schmid \cite{MillerSchmid2006} first established the Voronoi formula for $GL(3)$ over $\Q$. Then it was generalized to $GL(n)$ over $\Q$ in \cite{MillerSchmid2011} with another proof later found in \cite{GoldfeldLi2008}. In the appendix of \cite{MillerSchmid2011}, they showed an identity which is explicitly used by Jacquet, Piatetski-Shapiro, and Shalika to establish the functional equation of $L$-functions on $GL(n)\times GL(m)$ for $|n-m|>1$. In \cite{IchinoTemplier2013}, Ichino and Templier discovered such an identity as well. And they used it as a basic ingredient to prove their Voronoi formula over number fields.  Actually, this identity can be interpreted as an identity on the projection operator $\P_m^n$ for $1\le m\le n-2$ and $n\ge4$. We state it as follows.

\begin{proposition}\label{mainidentity}
	Let $n\ge4$. Let $F$ be a global field, and $\A$  the ring
	of adeles of $F$. Let  $\cA_{\text{cusp}}(GL_n)$ be the space of automorphic cusp forms
	on $GL_n(\A)$. Let $\P_m^n$ be the projection operator defined by (\ref{podfneq}), and  $\wt\P_m^n$  the dual projection operator defined by (\ref{dualpodfneq}). Then for $1\le m\le n-2$, restricted on $\cA_{\text{cusp}}(GL_n)$ the following identity holds:
	\begin{equation}\label{mainidentityeq}
		\P_m^n=\wt{\P}_m^n\circ\iota.
	\end{equation}
	Here $\iota$ is  defined by $\iota\varphi(g)=\varphi(g^\iota)=\varphi({}^tg^{-1})$ for  any $\varphi\in\cA_{\text{cusp}}(GL_n)$ and $g\in GL_n(\A)$.
\end{proposition}

Proposition~\ref{mainidentity} is essentially as same as Proposition A.1 in Miller and Schmid's article \cite{MillerSchmid2011}. 
In this note, we give a proof from the ideas in Cogdell and Piatetski-Shapiro's work \cite{CogdellPiatetski-Shapiro1999} on converse theorems for $GL(n)$. In \cite{CogdellPiatetski-Shapiro1999},  they defined two functions $U_\xi$ and $V_\xi$. The Fourier-Whittaker expansion of $\P_m^nU_\xi$ is well-known. Following the proof of \cite[Lemma 2.3]{CogdellPiatetski-Shapiro1999},  we will find the Fourier-Whittaker  expansion of $\P_m^nV_\xi$ in section~\ref{s:fwpo}, which is the main part of this note. Then we will use them to prove the identity (\ref{mainidentityeq}) in section~\ref{s:mainidpf}.

An interesting problem is whether one can use the identity (\ref{mainidentityeq}) to derive a Voronoi-type summation formula for $GL(n)\times GL(m)$ for $2\le m\le n-2$ as \cite{IchinoTemplier2013}. We hope it could be solved in the following years.

\section{The Fourier-Whittaker expansions of projections}\label{s:fw}

\subsection{The functions \texorpdfstring{$U_\xi$}{} and \texorpdfstring{$V_\xi$}{}}
Let $n\ge4$ be an integer. Let $F$ be a global field, $\A$ the ring of adeles of $F$, and $\psi$ a non-trivial additive character of $F\bs\A$.  Let $(\pi, V_\pi)$ be an irreducible generic representation
of $GL_n(\A)$. For any $\xi\in V_\pi$, let $W_\xi$ be the corresponding element
in the Whittaker model $\mathcal{W}(\pi,\psi)$ of $\pi$. Denote by $U_n$ the group of upper triangular matrices
in $GL_n$ with unit diagonal. We use the same letter $\psi$ to denote  the character on $U_n(\A)$  defined by $\psi(u)=\psi(u_{1,2}+\cdots+u_{n-1,n})$ for  $u=(u_{i,j})\in U_n(\A)$. 
For any $\xi$ in the space $V_\pi$ of $\pi$ and any $g\in GL_n(\A)$, set
\begin{align}
	U_\xi(g)&=\sum_{\gamma\in U_{n-1}(F)\bs GL_{n-1}(F)} W_\xi\of{\spmat{\gamma&\\
			&1}g},\label{udfn}\\
	V_\xi(g)&=\sum_{\gamma\in U_{n-1}(F)\bs GL_{n-1}(F)} W_\xi\of{\spmat{1&\\
			&\gamma}g}. \label{vdfn}
\end{align}
Then Cogdell and Piatetski-Shapiro \cite{CogdellPiatetski-Shapiro1999} used the following basic lemma as their approach to proving the converse theorems. 
\begin{lemma}[{\cite[Proposition 1]{CogdellPiatetski-Shapiro1999}}]\label{basiclemma}
	$\pi$ is automorphic cuspidal if and only if 
	\begin{equation}
		U_\xi(I_n)=V_\xi(I_n), \forall \xi\in V_\pi.
	\end{equation}
\end{lemma}

\begin{remark}
	For the projection operator $\P_{n-2}^n$ defined by (\ref{podfneq}), Cogdell and Piatetski-Shapiro \cite{CogdellPiatetski-Shapiro1999} actually showed that
	$\pi$ is automorphic cuspidal  if and only if
	\begin{equation}\label{convthmeq}
		\P_{n-2}^nU_\xi(I_n) =\P_{n-2}^nV_\xi(I_n)
	\end{equation}
	for all $\xi\in V_\pi$.
\end{remark}

\subsection{Projection operators}\label{podfns}

Let $1\le m\le n-2$, and let $Y_{n,m}$ be the unipotent radical of the standard parabolic subgroup of the partition
$(m+1,1,\dots,1)$ in $GL_n$. That is,  
$$Y_{n,m}=\set{\spmat{I_{m+1}&*&*&*\\&1&*&*\\&&\ddots&*\\&&&1\\}}\subset GL_{n}.$$

\begin{definition}\label{podfn}
	For $1\le m\le n-2$,  the {\it projection operator} $\P_m^n$ is defined by
	\begin{equation}\label{podfneq}
		\P_m^n\phi(g)=\int_{Y_{n,m}(F)\bs Y_{n,m}(\A)}\phi(yg){\psi}^{-1}(y)dy
	\end{equation}
	for a function $\phi$ on $GL_n(\A)$ which is left invariant under $Y_{n,m}(F)$. 
\end{definition}
\begin{remark}
	The definition~\ref{podfn} of $\P_m^n$ comes from Jacquet and Liu \cite[page 197]{JacquetLiu2020} and differs slightly from the definition used by Cogdell in \cite[\S2.2.1]{Cogdell2007}.
\end{remark}

Now, let $M_{(n-m-1)\times m}$ be the space of matrices of type $(n-m-1)\times m$, and 
\begin{equation}
	X_{n,m}=\set{\spmat{I_{m}&&\\x&I_{n-m-1}&\\&&1}: x\in M_{(n-m-1)\times m} }\subset GL_n. 
\end{equation}
Let
$w_n=\spmat{&&1\\&\adots&\\1&&}$
be the $n\times n$ matrix with unit opposite diagonal,  and
$w_{n,m}=\spmat{w_m&\\&w_{n-m}}.$

\begin{definition} \label{dualpodfn}
	The {\it dual projection operator}	$\wt{\P}_m^n$ is defined by
	\begin{equation}\label{dualpodfneq}
		\wt\P_m^n\phi(g)=\int_{ X_{n,m}(\A)}\Big[\int_{Y_{n,m}(F)\bs Y_{n,m}(\A)}\phi(yxw_{n,m}g^\iota)\psi^{-1}(y)dy\Big]dx
	\end{equation}
	for a function $\phi$ on $GL_n(\A)$ which is left invariant under $Y_{n,m}(F)$. Here for the convention of notations, we write $dg=dx$ for $g=\spmat{I_{m}&&\\x&I_{n-m-1}&\\&&1}\in X_{n,m}(\A)$ and $x\in M_{(n-m-1)\times m}(\A)$.
\end{definition}
\begin{remark}
	Our definition of $\wt\P_m^n$ differs from the definition of Cogdell \cite{Cogdell2007}. In \cite{Cogdell2007}, $\wt\P_m^n=\iota\circ\P_m^n\circ\iota$.
\end{remark}

\subsection{Fourier-Whittaker expansions of the projections}\label{s:fwpo}

Let $P_n=\spmat{GL_{n-1}&\ast\\&1}$, and $Q_n=\spmat{1&\ast\\&GL_{n-1}}$. The following proposition gives  the Fourier-Whittaker expansion of $\P_m^nU_\xi$.
\begin{proposition}[{\cite[Lemma 2.2]{Cogdell2007}}]\label{uprop}
	For any $\xi\in V_\pi$,
	$U_\xi$  is left invariant under $P_n(F)$ and 
	\begin{equation}\label{uproj}
		\P_m^nU_\xi(g)=\sum_{\gamma\in U_{m}(F)\bs GL_{m}(F)} W_\xi\of{\spmat{\gamma&\\
				&I_{n-m}}g}.
	\end{equation}
\end{proposition}

On the other hand,  we set
\begin{equation}\label{vmdfn}
	V_\xi^m(g):=V_\xi(\alpha_mg)
\end{equation}
for all $g\in GL_n(\A)$, where
$$\alpha_m=\spmat{0&1&0\\I_m&0&0\\0&0&I_{n-m-1}}.$$
Then $V_\xi^m$ is left invariant under $Q_m(F)$, where $Q_{m}=\alpha_m^{-1}Q_n\alpha_m$. To find the Fourier-Whittaker expansion of $\P_m^nV_\xi$, it suffices to find it for
for $\P_m^nV_\xi^m$.

First, we rewrite the expression of $V_\xi^m(g)$ as  the same form as $U_\xi(g)$. Let	$\wt W_\xi(g):=W_\xi(w_ng^\iota)$ be the dual of $W_\xi$.
Then $\wt W_\xi\in \mathcal{W}(\wt{\pi},{\psi}^{-1})$, where $\wt{\pi}$ is the contragradient representation of $\pi$. 
\begin{lemma}\label{pvlem}
	In terms of $\wt{W}_\xi(g)$, $V_\xi^m(g)$ can be rewritten as follows
	\begin{equation}\label{pvlemeq}
		V_\xi^m(g)=\sum_{\gamma\in U_{n-1}(F)\bs GL_{n-1}(F)}\wt W_\xi\of{\spmat{\gamma&\\&1}w_n\alpha_m g^\iota}.
	\end{equation}
\end{lemma}
\begin{proof} Here we follow the ideas in the proof of the expansion of $V_\xi(g)$ in \cite[\S3.2]{BookerKrishnamurthy2016}. 
	Notice first that $$\spmat{1&\\&\gamma}=\alpha_{n-1}\spmat{\gamma&\\&1}\alpha_{n-1}^{-1}, W_\xi(g)=\wt W_\xi(w_ng^\iota)  \text{ and }\alpha_{n-1}=w_n\spmat{w_{n-1}&\\&1}.$$ Then by the definition of $V_\xi^m$,  we have
	\begin{align}
		V^m_\xi(g)&=\sum_{\gamma\in U_{n-1}(F)\bs GL_{n-1}(F)}W_\xi\of{\spmat{1&\\&\gamma} \alpha_m g}\nonumber\\
		&=\sum_{\gamma\in U_{n-1}(F)\bs GL_{n-1}(F)}W_\xi\of{\alpha_{n-1}\spmat{\gamma&\\&1}\alpha_{n-1}^{-1} \alpha_m g}\nonumber\\
		&=\sum_{\gamma\in U_{n-1}(F)\bs GL_{n-1}(F)}\wt W_\xi\of{w_n \left[\alpha_{n-1}\spmat{\gamma&\\&1}\alpha_{n-1}^{-1} \alpha_m g\right]^\iota}\nonumber\\
		&=\sum_{\gamma\in U_{n-1}(F)\bs GL_{n-1}(F)}\wt W_\xi\of{w_n \alpha_{n-1}\spmat{\gamma^\iota&\\&1}\alpha_{n-1}^{-1} \alpha_m g^\iota}\nonumber\\
		&=\sum_{\gamma\in U_{n-1}(F)\bs GL_{n-1}(F)}\wt W_\xi\left(\spmat{w_{n-1}\gamma^\iota w_{n-1}^{-1}&\\&1} w_n
		\alpha_m g^\iota\right)\nonumber\\
		&=\sum_{\gamma\in U_{n-1}(F)\bs GL_{n-1}(F)}\wt W_\xi\of{\spmat{\gamma&\\&1}w_n\alpha_m g^\iota}.
	\end{align}
	Here, the last equality follows since the transformation $\gamma\mapsto w_{n-1}\gamma^\iota w_{n-1}^{-1}$ permutes
	the set of right cosets $U_{n-1}(F)\gamma$ in $GL_{n-1}(F)$.
\end{proof}

Now, we  follow the ideas in the proof of \cite[Lemma 2.3]{CogdellPiatetski-Shapiro1999} to deduce the Fourier-Whittaker expansion of $\P_m^nV_\xi^m$.

\begin{proposition}\label{vprop}
	For $V_\xi^m$ defined in (\ref{vmdfn}),  we have that
	\begin{equation}\label{vproj}
		\P_m^nV_\xi^m(g)=\sum_{\gamma\in U_{m}(F)\bs GL_{m}(F)}  \int_{  X_{n,m}(\A)}
		\wt{W}_\xi\left(\spmat{\gamma &\\&I_{n-m}}x w_{n,m} g^\iota\right) dx.	
	\end{equation}
\end{proposition}

\begin{proof}	
	By the definition of $\P_m^n$ in (\ref{podfneq}), we have
	\begin{equation}\label{pvdfn}
		\P_m^nV_\xi^m(g)=\int_{Y_{n,m}(F)\bs Y_{n,m}(\A)}V_\xi^m(yg)\psi^{-1}(y)dy.
	\end{equation}
	Then by Lemma \ref{pvlem}, we can rewrite $V_\xi^m(yg)$ as follows:
	\begin{align}
		&V_\xi^m(yg)\nonumber\\
		&= \sum_{\gamma\in U_{n-1}(F)\bs GL_{n-1}(F)}\wt{W}_\xi\of{\spmat{\gamma&\\&1}w_n\alpha_m (yg)^\iota}\nonumber\\
		&= \sum_{\gamma\in U_{n-1}(F)\bs GL_{n-1}(F)}\wt{W}_\xi\of{\spmat{\gamma&\\&1}\cdot w_n\alpha_m \spmat{w_m&\\&w_{n-m}}^{-1}\cdot \spmat{w_m&\\&w_{n-m}} y^\iota g^\iota}\nonumber\\
		&= \sum_{\gamma\in U_{n-1}(F)\bs GL_{n-1}(F)}\wt{W}_\xi\of{\spmat{\gamma&\\&1} \spmat{w_m&\\&w_{n-m}} y^\iota g^\iota},
	\end{align}
	where 
	$$w_n\alpha_m \spmat{w_m&\\&w_{n-m}}^{-1}:=\spmat{\gamma_0&\\&1}$$
	and $\gamma_0\in GL_{n-1}(F)$ is absorbed by $\gamma$ into the sum.
	
	For $y$, we have the following decomposition:
	\begin{equation}
		y=\spmat{I_m&\\&u}\spmat{I_m&0&y'\\0&1&0\\0&0&I_{n-m-1}},
	\end{equation}
	where $u\in U_{n-m}(F\bs\A)$ and $y'\in M_{m\times(n-m-1)}(F\bs\A)$. Here we use the notations $U_{n}(F\bs\A)$ and $M_{m\times\ell}(F\bs\A)$ as shorthand for $U_n(F)\bs U_n(\A)$ and $M_{m\times\ell}(F)\bs M_{m\times\ell}(\A)$, respectively. It follows that 
	\begin{equation}
		\spmat{w_m&\\&w_{n-m}}y^\iota \spmat{w_m&\\&w_{n-m}}^{-1}=\spmat{I_m&\\&w_{n-m}u^\iota w_{n-m}^{-1}}\spmat{I_{m}&&\\x'&I_{n-m-1}&\\&&1},
	\end{equation}
	where $x'\in M_{(n-m-1)\times m}(F\bs\A)$ is a permutation of $-{}^ty'$. 
	
	Notice that $\psi(y)=\psi^{-1}(w_{n-m}u^\iota w_{n-m}^{-1})$. 
	Making change of variables $u\mapsto w_{n-m}u^\iota w_{n-m}^{-1}$ and $y'\mapsto x'$, we get 
	\begin{multline}\label{pvdfn2}
		\P_m^nV_\xi^m(g)=\int_{x'\in  M_{(n-m-1)\times m}(F\bs\A)}\int_{u\in U_{n-m}(F\bs\A)}\sum_{\gamma\in U_{n-1}(F)\bs GL_{n-1}(F)}\\
		\wt{W}_\xi\of{\spmat{\gamma&\\&1}  \spmat{I_m&\\&u}\spmat{I_{m}&&\\x'&I_{n-m-1}&\\&&1} \spmat{w_m&\\&w_{n-m}} g^\iota}\psi(u)dudx'.
	\end{multline}
	
	Next, we integrate over the last column of $u$.  
	
	Decompose $u$ into the product of three matrices:
	\begin{equation}\label{vprojmd}
		\spmat{I_m&\\&u}=\spmat{I_{n-2}&&\\
			&1&u_{n-1}\\&&1}\spmat{I_m&&&&&\\&1&&&&u_{m+1}\\&&\ddots&&&\vdots\\&&&1&&u_{n-2}\\&&&&1&0\\&&&&&1}\spmat{I_m&&\\&u'&\\&&1},
	\end{equation}
	where $u'\in U_{n-m-1}(F\bs\A)$  satisfies $\psi(u)=\psi(u_{n-1})\psi(u')$.
	
	First, by $\spmat{\gamma &\\&1}\spmat{I_{n-2}&&\\&1&u_{n-1}\\&&1} =\spmat{I_{n-2}&&*\\&1&\gamma_{n-1,n-1}u_{n-1}\\&&1} \spmat{\gamma &\\&1}$, we have
	\begin{equation}
		\wt{W}_\xi\left(\spmat{\gamma &\\&1}\spmat{I_m&\\&u}g\right)=\psi(-\gamma_{n-1,n-1}u_{n-1})\wt{W}_\xi\left(\spmat{\gamma &\\&1}\spmat{I_m&&&&&\\&1&&&&u_{m+1}\\&&\ddots&&&\vdots\\&&&1&&u_{n-2}\\&&&&1&0\\&&&&&1}\spmat{I_m&&\\&u'&\\&&1}g\right).	
	\end{equation}
	
	Integrating over $u_{n-1}$ on $F\bs\A$, we get $\gamma_{n-1,n-1}=1$. So we can write $\gamma$ as
	\begin{equation}
		\gamma=\spmat{\gamma'&\\&1}\spmat{I_{n-2}&\\v&1}
	\end{equation}
	with $\gamma'\in U_{n-2}(F)\bs GL_{n-2}(F)$ and $v=(v_1,\dots,v_{n-2})\in F^{n-2}$. Then
	\begin{equation}
		\spmat{I_{n-2}&&\\v&1&\\&&1}\spmat{I_m&&&&&\\&1&&&&u_{m+1}\\&&\ddots&&&\vdots\\&&&1&&u_{n-2}\\&&&&1&0\\&&&&&1}=\spmat{I_m&&&&&\\&1&&&&u_{m+1}\\&&\ddots&&&\vdots\\&&&1&&u_{n-2}\\&&&&1&u_{m+1}v_{m+1}+\cdots+u_{n-2}v_{n-2}\\&&&&&1}\spmat{I_{n-2}&&\\v&1&\\&&1}.
	\end{equation}
	This implies that
	\begin{multline}
		\wt{W}_\xi\left(\spmat{\gamma &\\&1}\spmat{I_m&&&&&\\&1&&&&u_{m+1}\\&&\ddots&&&\vdots\\&&&1&&u_{n-2}\\&&&&1&0\\&&&&&1}\spmat{I_m&&\\&u'&\\&&1}g\right)\\
		=\psi(-u_{m+1}v_{m+1}-\cdots-u_{n-2}v_{n-2})\wt{W}_\xi\left(\spmat{\gamma' &\\&I_2}\spmat{I_{n-2}&&\\v&1&\\&&1}\spmat{I_m&&\\&u'&\\&&1}g\right).
	\end{multline}
	Integrating over $u_{m+1},\dots,u_{n-2}$, we get  $v_{m+1}=\cdots=v_{n-2}=0$.  Thus, after integrating  over the last column of $u$, we get that
	\begin{multline}\label{pvfeeq3}
		\P_m^nV_\xi^m(g)=\int_{u'\in  U_{n-m-1}(F\bs\A)}\int_{x'\in  M_{(n-m-1)\times m}(F\bs\A)}\sum_{v\in F^{m}}\sum_{\gamma'\in U_{n-2}(F)\bs GL_{n-2}(F)}\\
		\wt{W}_\xi\Bigg(\spmat{\gamma' &\\&I_2}\spmat{I_{m}&&&\\0&I_{n-m-2}&&&\\v&0&1&\\&&&1}\spmat{I_m&&\\&u'&\\&&1}\spmat{I_{m}&&\\x'&I_{n-m-1}&\\&&1}
		\spmat{w_m&\\&w_{n-m}} g^\iota\Bigg) dx'\psi(u')du'.
	\end{multline}
	
	Notice that
	\begin{equation}
		\spmat{I_m&&\\&u'&\\&&1}\spmat{I_{m}&&\\x'&I_{n-m-1}&\\&&1}=\spmat{I_{m}&&\\u'x'&I_{n-m-1}&\\&&1}\spmat{I_m&&\\&u'&\\&&1}.
	\end{equation}
	By a change of variables $u'x'\mapsto x'$ and collapsing $v$ with the last row of $x'$, 
	\begin{multline}
		\P_m^nV_\xi^m(g)=\int_{u'\in  U_{n-m-1}(F\bs\A)} \int_{x''\in  M_{(n-m-2)\times m}(F\bs\A)}\int_{x_{n-1}\in \A^{m}}\\
		\sum_{\gamma\in U_{n-2}(F)\bs GL_{n-2}(F)}
		\wt{W}_\xi\left(\spmat{\gamma &\\&I_{2}}\spmat{I_m&&&\\x''&I_{n-m-2}&&\\x_{n-1}&0&1&\\&&&1}\spmat{I_m&&\\&u'&\\&&1}\spmat{w_m&\\&w_{n-m}} g^\iota\right)  dx_{n-1}dx''\psi(u')du'.
	\end{multline}
	Then interchanging the matrices on $(x'',x_{n-1})$ and $u'$ and making a change of variables again, we get that
	\begin{multline}\label{vprojeq2}
		\P_m^nV_\xi^m(g)=\int_{u'\in  U_{n-m-1}(F\bs\A)} \int_{x''\in  M_{(n-m-2)\times m}(F\bs\A)}\int_{x_{n-1}\in \A^{m}}
		\sum_{\gamma\in U_{n-2}(F)\bs GL_{n-2}(F)}\\
		\wt{W}_\xi\left(\spmat{\gamma &\\&I_{2}}\spmat{I_m&&\\&u'&\\&&1}\spmat{I_m&&&\\x''&I_{n-m-2}&&\\x_{n-1}&0&1&\\&&&1}\spmat{w_m&\\&w_{n-m}} g^\iota\right) dx_{n-1}dx''\psi(u')du'.
	\end{multline}

	Iterating the previous steps on $u'$ from (\ref{vprojmd}) to (\ref{vprojeq2}) until getting rid of $u'$, we can finally reach that
	\begin{multline}\label{}
		\P_m^nV_\xi^m(g)=\sum_{\gamma\in U_{m}(F)\bs GL_{m}(F)}  \int_{  M_{(n-m-1)\times m}(\A)}
		\wt{W}_\xi\left(\spmat{\gamma &\\&I_{n-m}}\spmat{I_{m}&&\\x&I_{n-m-1}&\\&&1}\spmat{w_m&\\&w_{n-m}} g^\iota\right) dx,
	\end{multline}
	which is (\ref{vproj}). This completes the proof of Proposition~\ref{vprop}.
\end{proof}

\section{Proof of Proposition~\ref{mainidentity}}\label{s:mainidpf}

In this section, we prove Proposition~\ref{mainidentity} by the Fourier-Whittaker expansions of $\P_m^nU_\xi$ and $\P_m^nV_\xi$.

Suppose $\varphi\in \cA_{\text{cusp}}(GL_n)$ is an automorphic cusp forms om $GL_n(\A)$ and $g\in GL_n(\A)$. By Lemma~\ref{basiclemma}, we have $\varphi(g)=U_\varphi(g)=V_\varphi(g)$. Since $\varphi$ is automorphic, we have $V_\varphi(g)=V^m_\varphi(g)$ as well. It follows that
\begin{equation}\label{s:idpfeq1}
	\P_m^n\varphi(g)=\P_m^nU_\varphi(g)=\P_m^nV^m_\varphi(g).
\end{equation}
Let $\wt\varphi=\iota\varphi$, then $\wt{W}_\varphi\in\mathcal{W}(\wt{\pi},\psi^{-1})$ is the Whittaker function of $\wt\varphi$. By Proposition~\ref{uprop}, we get the Fourier-Whittaker expansion of $\P_m^n\wt\varphi(g)$ as follows
\begin{equation}\label{s:idpfeq2}
	\P_m^n\wt\varphi(g)=\sum_{\gamma\in U_{m}(F)\bs GL_{m}(F)} \wt{W}_\varphi\of{\spmat{\gamma&\\
			&I_{n-m}}g}.
\end{equation}
On the other hand, by the definition (\ref{dualpodfneq}) of $\wt\P_m^n\wt\varphi(g)$, we have
\begin{equation}\label{s:idpfeq3}
	\wt\P_m^n\wt\varphi(g)=\int_{  X_{n,m}(\A)} \P_m^n\wt\varphi(x w_{n,m} g^\iota)dx.
\end{equation}
Plugging (\ref{s:idpfeq2}) into (\ref{s:idpfeq3})  and interchanging the sum and the integration, we get  that
\begin{equation}\label{s:idpfeq4}
	\wt\P_m^n\wt\varphi(g)= \sum_{\gamma\in U_{m}(F)\bs GL_{m}(F)}  \int_{  X_{n,m}(\A)}
	\wt{W}_\varphi\left(\spmat{\gamma &\\&I_{n-m}}x w_{n,m} g^\iota\right) dx.
\end{equation}
By Proposition~\ref{vprop}, the right hand side of (\ref{s:idpfeq4}) is equal to the Fourier-Whittaker expansion of $\P_m^nV^m_\varphi(g)$, which means that
\begin{equation}\label{s:idpfeq5}
	\wt\P_m^n\wt\varphi(g)=\P_m^nV^m_\varphi(g).
\end{equation}
Hence, from (\ref{s:idpfeq1}) and (\ref{s:idpfeq5}) we get that
\begin{equation}
	\P_m^n\varphi(g)=\wt\P_m^n\wt\varphi(g),
\end{equation}
which gives us the desired identity (\ref{mainidentityeq}). This completes the proof of Proposition~\ref{mainidentity}.

\begin{remark}
	From the proof above, one can see that the identity (\ref{mainidentityeq}) in Proposition~\ref{mainidentity} corresponds to Cogdell's identity $\P_m^nU_\xi(I_{m+1})=\P_m^nV_\xi(I_{m+1})$ in \cite[Proposition 5.1]{Cogdell2007}, which is related to the functional equation of the Eulerian integrals \cite[Theorem 2.1]{Cogdell2007} for $GL(n)\times GL(m)$.
\end{remark}

\section{Acknowledgements}
The author would like to thank his advisor Professor Xiaoqing Li for her guidance on Cogdell and Piatetski-Shapiro's converse theorems.  He is also grateful to Pan Yan, Liyang Yang, and Shaoyun Yi for the helpful discussions.

\end{document}